\documentclass[a4paper, 12pt]{article}
\usepackage{amssymb, amsthm,amsmath}
\usepackage{enumerate}
\usepackage{stmaryrd}
\usepackage{mathrsfs}
\PassOptionsToPackage{hyphens}{url} \usepackage{hyperref}
\usepackage[all]{xy}

\makeatletter
\def\blfootnote{\gdef\@thefnmark{}\@footnotetext}
\makeatother

\numberwithin{equation}{section}
\newtheorem{thm}[equation]{Theorem}
\newtheorem{lem}[equation]{Lemma}
\newtheorem{prop}[equation]{Proposition}

\newtheorem{introthm}{Theorem}

\theoremstyle{definition}

\theoremstyle{remark}
\newtheorem{rem}[equation]{Remark}

\newtheorem{example}[equation]{Example}

\newcommand{\C}{\mathcal{C}}
\newcommand{\D}{\mathcal{D}}
\newcommand{\E}{\mathcal{E}}
\newcommand{\K}{\mathcal{K}}

\newcommand{\U}{\mathcal{U}}

\newcommand{\id}{\operatorname{id}}
\newcommand{\coker}{\operatorname{coker}}

\newcommand{\coim}{\operatorname{coim}}

\newcommand{\colim}{\operatorname{colim}}
\newcommand{\hocolim}{\operatorname{hocolim}}
\newcommand{\disc}{\operatorname{disc}}

\newcommand{\sd}{\operatorname{sd}}
\newcommand{\Ex}{\operatorname{Ex}}

\newcommand{\Sing}{\operatorname{Sing}}

\begin{document}

\title{Eilenberg--Mac Lane Spaces for Topological Groups}
\author{Ged Corob Cook\thanks{This work was done partly while supported by EPSRC grant EP/N007328/1, and partly while supported by ERC grant 336983 and the Basque government grant IT974-16.}}
\date{}
\renewcommand\footnotemark{}
\maketitle

\begin{abstract}
The goal of this paper is to establish a topological version of the notion of an Eilenberg--Mac Lane space. If $X$ is a pointed topological space, $\pi_1(X)$ has a natural topology coming from the compact-open topology on the space of maps $S^1 \to X$. In general the construction does not produce a topological group because it is possible to create examples where the group multiplication $\pi_1(X) \times \pi_1(X) \to \pi_1(X)$ is discontinuous. This failure to obtain a topological group has been noticed by others, for example Fabel. However, if we work in the category of compactly generated, weakly Hausdorff spaces, we may retopologise both the space of maps $S^1 \to X$ and the product $\pi_1(X) \times \pi_1(X)$ with compactly generated topologies to get that $\pi_1(X)$ is a group object in this category. Such group objects are known as $k$-groups.

Next we construct the Eilenberg--Mac Lane space $K(G,1)$ for any totally path-disconnected $k$-group $G$. The main point of this paper is to show that, for such a $G$, $\pi_1(K(G,1))$ is isomorphic to $G$ in the category of $k$-groups.

All totally disconnected locally compact groups are $k$-groups and so our results apply in particular to profinite groups. This answers questions that have been raised by Sauer.

We also show that there are Mayer--Vietoris sequences and a Seifert--van Kampen theorem in this theory.

The theory requires a careful analysis using model structures and other homotopical structures on cartesian closed categories as we shall see that no theory can be comfortably developed in the classical world.
\end{abstract}

\blfootnote{2010 Mathematics Subject Classification: 55P20; 18G55; 18G30.}
\blfootnote{Keywords: Eilenberg--Mac Lane space; $k$-group; homotopical algebra.}

\section*{Introduction}


There are three constructions for discrete groups $G$ that define the same classifying space $BG$. This space is an Eilenberg--Mac Lane space $K(G,1)$; it represents cohomology $H^1(X,G)$; it classifies $G$-torsors. This is described in detail in the preface to \cite{LurieHTT}.

For topological groups, the third of these definitions makes sense, and indeed there are standard constructions of $BG$ for topological groups $G$ that satisfy this definition. In fact there are two slightly different definitions of classifying space possible in this context. A $G$-torsor is a map $X \to Y$ of spaces together with a free $G$-action on $X$ such that $X/G \cong Y$; some definitions require also that the $G$-action be locally trivial. With this local triviality condition, the classifying space for $G$-torsors is given by Milnor's construction. An alternative classifying space contruction is given by Segal in \cite{Segal}: by considering $G$ as a one-object topological category, we may take the geometric realisation of the nerve of $G$. The question of what this space (usually also denoted $BG$) actually classifies is rather more subtle; see \cite{Moerdijk} for details on this.

On the other hand, there has not previously been any way to apply the first two definitions to the situation of topological groups. Here we construct Eilenberg--Mac Lane spaces for totally path-disconnected groups, in Theorem \ref{EMspace}. Being totally path-disconnected is crucial because the lack of any `higher-dimensional homotopy' for such spaces allows us to control the homotopy of the new spaces we will construct from them.

The paper is structured as follows. Section \ref{kmodules} gives some background on the category of compactly generated, weakly Hausdorff spaces and group and module objects in this category.

Section \ref{tophomotopy} motivates our definition of homotopy groups and related structures, over a more common definition in the literature. For a pointed topological space $X$, there is a natural compact-open topology on the space of continuous pointed maps $S^1 \to X$, and we can give the fundamental group $\pi_1(X)$ the quotient topology from this. In general the construction does not produce a topological group because it is possible to create examples where the product map $\pi_1(X)\times \pi_1(X)\to\pi_1(X)$ is discontinuous. This failure to obtain a topological group has been noticed by other authors including for example Fabel. However, the novelty here is that our $\pi_1(X)$ is a group object in the category of compactly generated weakly Hausdorff topological spaces. Such groups are called $k$-groups. In particular, it is always the case that the group multiplication is continuous when one retopologises the product $\pi_1(X)\times\pi_1(X)$ with the compactly generated topology. This raises the prospect of constructing a topological space $X$ with chosen homotopy groups.

We take advantage of this better categorical behaviour by largely restricting attention in this paper to the category $\U$ of compactly generated, weakly Hausdorff spaces.

Section \ref{model} recalls the model category theory and homotopical structures will need from \cite{Myself}. We have two different indispensable homotopical structures on the category of simplicial spaces $s\U$, which we call the compact Hausdorff structure and the regular structure; these are defined below. Weak equivalences in the former structure are also weak equivalences in the latter, while weak equivalences in the regular structure induce isomorphisms of the topological homotopy groups described above.

We give here definitions for the singular simplicial space functor $\underline{\Sing}: \U \to s\U$ and geometric realisation $\lvert - \rvert: s\U \to \U$. These form an adjoint pair analogously to the classical singular simplicial set functor and geometric realisation. Understanding these functors, together with the left derived functor $L\lvert - \rvert$ of $\lvert - \rvert$, is crucial to our approach.

Sections \ref{Sing} and \ref{excision} are the technical heart of the paper. The former consists primarily of proving a continuous analogue of the Seifert--van Kampen Theorem:

\begin{introthm}[Theorem \ref{singhoco}]
If $C$ is an open cover of $X \in \U$, write $C'$ for the poset of finite intersections of sets in $C$, ordered by inclusion. Then $\underline{\Sing}(X)$ is weakly equivalent (in the regular structure on $s\U$) to the homotopy colimit (in the compact Hausdorff structure on $s\U$) of $\{\underline{\Sing}(U)\}_{U \in C'}$.
\end{introthm}

In the latter, analogously to our topological homotopy groups, we consider the topological singular homology groups of a space, and prove an Excision Theorem and a Mayer--Vietoris sequence.

\begin{introthm}[Theorem \ref{exc}]
Given subspaces $A \subseteq B \subseteq X$ in $\U$ with $A$ closed and $B$ open, the inclusion $(X \setminus A, B \setminus A)֓ \to (X,B)$ induces isomorphisms of the homology group objects $H_n(X \setminus A, B \setminus A)֓ \to H_n(X,B)$ for all $n$. Equivalently, for open subspaces $A, B \subseteq X$ covering $X$, the inclusion $(B,A \cap B)֓ \to (X,A)$ induces isomorphisms of homology group objects $H_n(B,A \cap B) \to H_n(X,A)$ for all $n$.
\end{introthm}

\begin{introthm}[Theorem \ref{MV}]
For open subspaces $A, B \subseteq X$ covering $X$ there is a long exact sequence of homology group objects $$\cdots \to H_{n+1}(X) \to H_n(A \cap B) \to H_n(A) \oplus H_n(B) \to H_n(X) \to \cdots.$$
\end{introthm}

We also show in Theorem \ref{generalised} that this topological singular homology theory satisfies the axioms of a generalised homology theory.

Finally in Section \ref{EM}, we put all the work together to construct the promised Eilenberg--Mac Lane spaces. Specifically, we show:

\begin{introthm}[Theorem \ref{tpd}]
\label{introtpd}
Suppose $X \in s\U$ with $X_n$ totally path-disconnected for all $n$. Then $\underline{\Sing} \circ L\lvert X \rvert$ is weakly equivalent to $X$ in the regular structure.
\end{introthm}

Now given a group object $G$ in $\U$, in \cite{Myself} a simplicial space $\bar{W}G$ is constructed with $\pi_1(\bar{W}G) = G$ and all other homotopy groups trivial. When $G$ is totally path-disconnected so is each $\bar{W}G_n$. So we can apply Theorem \ref{introtpd}.

\begin{introthm}[Theorem \ref{EMspace}]
If $G$ is totally path-disconnected, $L\lvert \bar{W}G \rvert$ is an Eilenberg--Mac Lane space $K(G,1)$ for $G$.
\end{introthm}

Finally we justify our use of the totally path-disconnected condition, rather than the stronger one of being totally disconnected, by constructing in Example \ref{freepseudo} a totally path-disconnected group object in $\U$ which is not totally disconnected.

\section{Topological groups and modules}
\label{kmodules}

We work in $\U$, the category of compactly generated, weakly Hausdorff spaces, together with a modification of the compact-open topology on spaces of maps. This makes $\U$ into a cartesian closed category; see \cite{Strickland} for a good general reference on this. We will write $\underline{\U}(X,Y)$ for the space of maps $X \to Y$ in $\U$ equipped with this topology, and more generally for a category $\C$ enriched over a category $\D$ we will write $\C_\D(X,Y)$ for the enriched hom-object, or $\underline{\C}(X,Y)$ when there is no ambiguity.

Note that the definitions in this section also make sense for spaces in $\K$, the category of compactly generated spaces, which is also cartesian closed; we restrict to $\U$ for compatibility with later sections, where a model structure is only constructed on $\U$.

We can define internal group objects in $\U$: these are groups $G$ with a topology on their underlying set making multiplication $G \times G \to G$ and inversion $G \to G$ continuous. Note that the product $G \times G$ here is the internal product in $\U$. For ease of use we will refer to such group objects as topological groups, although they are not topological groups in general: in particular, writing $\times_0$ for the product in the category of topological spaces, the map $G \times_0 G \to G$ may not be continuous. Indeed, \cite[Example 2.14]{LaMartin} shows that group objects $G$ in $\U$ for which $G \times_0 G \to G$ is not continuous do arise in nature.

As for these topological groups, we can define a category of ring objects in $\U$, which we will call topological rings. Similarly, for a topological ring $R$, we write $R$-$\U Mod$ for the category of left $R$-module objects in $\U$. All this includes as a special case $\U Ab$, the category of abelian group objects in $\U$, which is $\mathbb{Z}$-$\U Mod$ where $\mathbb{Z}$ is given the discrete topology -- indeed, results that apply analogously to $R$-$\U Mod$ will mostly be stated in this paper as results about $\U Ab$, with the generalisation being left to the reader. For details on this module category, see \cite[Sections 7-8]{Myself}. In the rest of this section we summarise the results we will need, without further reference.

Given $A,B \in R$-$\U Mod$, write $\U_R(A,B)$ for the set of morphisms $A \to B$: this naturally has the structure of an abelian group, and the restriction of the topology on $\underline{\U}(A,B)$ makes $\underline{\U}_R(A,B)$ into an abelian topological group, so that $R$-$\U Mod$ becomes an additive category enriched over $\U Ab$.

One of the generalisations of abelian categories is the concept of quasi-abelian categories. A quasi-abelian category is an additive category with all kernels and cokernels, satisfying two additional properties:
\begin{enumerate}[(i)]
\item in any pull-back square
\[
\xymatrix{A' \ar[r]^{f'} \ar[d] & B' \ar[d] \\
A \ar[r]^{f} & B,}
\]
if $f$ is the cokernel of some map then so is $f'$;
\item in any push-out square
\[
\xymatrix{A \ar[r]^{f} \ar[d] & B \ar[d] \\
A' \ar[r]^{f'} & B',}
\]
if $f$ is the kernel of some map then so is $f'$.
\end{enumerate}
It turns out that $R$-$\U Mod$ is a complete and cocomplete quasi-abelian category. $R$-$\U Mod$ also has free modules. That is, the forgetful functor $R$-$\U Mod \to \U$ has a left adjoint, which we will write as $F$.

See \cite[Section 1]{Schneiders} for a complete account of homological algebra over quasi-abelian categories. For $\mathcal{C}$ a quasi-abelian category there is an abelian category $\mathcal{LH(C)}$ called the left heart of $\mathcal{C}$ whose objects consist of monomorphisms in $\mathcal{C}$. There is a fully faithful, exact embedding $\mathcal{C} \to \mathcal{LH(C)}$ which induces an equivalence on the derived categories of chain complexes in the two categories. Thus for a chain complex in $\mathcal{C}$ we can think of it as a chain complex in $\mathcal{LH(C)}$ and take homology there: the resulting $n$th homology functor sends $$\cdots \to C_{n+1} \xrightarrow{f_n} C_n \xrightarrow{f_{n-1}} C_{n-1} \to \cdots$$ to the map $\coim(f_n) \to \ker(f_{n-1})$ in the left heart. Since the embedding $\mathcal{C} \to \mathcal{LH(C)}$ is exact, the family $LH_\ast$ of functors takes a short exact sequence of chain complexes in $\mathcal{C}$ and gives a long exact sequence in the left heart.

Note that all this can be dualised to give a right heart and a right homology functor, which we will not use in this paper.

The class of all kernel-cokernel pairs in $R$-$\U Mod$ thus makes it into an exact category, in the sense of Quillen. We will refer to this as the quasi-abelian structure, or the regular structure, by analogy to the non-additive case below.

However, with this exact structure $R$-$\U Mod$ does not have enough projectives, which is a serious drawback in doing homological algebra. Since $R$-$\U Mod$ has free modules (that is, the forgetful functor to $\U$ has a left adjoint), we have various other structures available, which make $R$-$\U Mod$ into a left exact structure (a generalisation of an exact structure which requires a well-behaved class of deflations but not inflations) with enough projectives. We are interested here the case where the projectives are summands of free modules on disjoint unions of compact Hausdorff spaces and the deflations are the epimorphisms $A \to B$ such that the induced map $\U_R(P,A) \to \U_R(P,B)$ is surjective for all projectives $P$. We will refer to this as the compact Hausdorff exact structure.

\section{Topological homotopy groups}
\label{tophomotopy}

Given a space $X$, a classical homotopy group, which we write as $\pi^{abs}_n(X,x)$, is calculated as a quotient of the set of pointed maps $S^n \to X$. Since this set of maps has a natural topology, the compact-open topology, the obvious approach to putting a topology on $\pi^{abs}_n(X,x)$ is to give it the quotient topology. Indeed there is a literature using this definition; see for example \cite{Fabel}.

This definition has a shortcoming: this topology does not always make $\pi^{abs}_n(X,x)$ into a topological group, as shown in \cite{Fabel}. Essentially the problem is caused by the failure of the category of topological spaces, together with the compact-open topology, to be cartesian closed. Our solution is to restrict to a convenient category of spaces which is.

Using the same definitions internally to $\K$, the category of compactly generated space, $\pi^{abs}_n(X,x)$ is made into a group object in $\K$, as shown in \cite{Myself}. We write $\pi^\K_n(X,x)$ for this. But this definition has a shortcoming too. Given a fibration in $\U$ (in the model structure defined below), we would like to emulate the classical situation by obtaining a long exact sequence of topological homotopy groups; but such a sequence does not hold here. The problem here is that $\U$ and $\K$ are regular, but not Barr-exact, categories. See \cite{Myself} for more on this. But regular categories $C$ have a well-behaved, canonical embedding into a Barr-exact category: the objects of this category are equivalence relations in $C$. We will identify spaces in $\U$ with their image in the exact completion $\U_{ex}$ under this embedding.

Explicitly, given $X \in \U$, the space $\underline{\U}(S^n,X)$ of based maps $S^n \to X$ (for some choice of basepoint $x$) has a natural topology in $\U$, as before; so does the space $H$ of based homotopies between such maps. The inclusions $$(\id,0), (\id,1): S^n \to S^n \times [0,1]$$ induce a pair of maps $H \rightrightarrows \underline{\U}(S^n,X)$, or equivalently a map $$H \to \underline{\U}(S^n,X) \times \underline{\U}(S^n,X).$$ Factoring this map as a quotient followed by an injection gives a injective map $$H' \to \underline{\U}(S^n,X) \times \underline{\U}(S^n,X),$$ and this is the equivalence relation we take to be the $n$th homotopy group $\pi_n(X,x)$: it is a group object in $\U_{ex}$.

With this definition, a fibration in $\U$, in the regular structure defined below, does give a long exact of these homotopy groups: \cite[Section 5]{Myself}.

\section{Model structures}
\label{model}

In this paper we want to construct spaces with some chosen topological homotopy group $G$. As for abstract groups, it quickly becomes necessary to develop some tools to allow the calculation of the homotopy groups of any interesting spaces. In our context the most important tools are the model structures defined in \cite{Myself}.

We recall here the pieces of machinery we will need. We use the definitions of model structures and model categories given in \cite[Section 1.1]{Hovey}. That is, a model category is a complete and cocomplete category together with a model structure, and factorisations are required to be functorial.

Given two model categories $\C$ and $\D$ and an adjoint pair of functors $F: \C \rightleftarrows \D: G$ with $F \dashv G$, suppose $F$ preserves cofibrations and trivial cofibrations, or equivalently $G$ preserves fibrations and trivial fibrations. We say that such a pair of functors forms a Quillen adjunction. Then we can define the left derived functor $LF$ of $F$ and dually the right derived functor $RG$ of $G$, such that these derived functors preserve weak equivalences and hence induce functors on the associated homotopy categories. Explicitly, $LF$ can be constructed as the composite of $F$ with the cofibrant replacement functor on $\C$, and dually for $RG$.

An important example of this idea is that of homotopy colimits (or dually homotopy limits): given a small category $\E$, the colimit functor $\colim: \C^\E \to \C$ is left adjoint to the diagonal functor $\Delta: \C \to \C^\E$. We would like to put a model structure on $C^\E$ that makes these functors into a Quillen adjunction.

\begin{thm}
If $\C$ is class-cofibrantly generated in the sense of \cite{Myself}, the projective model structure defined in \cite[Section A.2.8]{LurieHTT} exists on $\C^\E$, is class-cofibrantly generated, and makes $\colim \dashv \Delta$ a Quillen adjunction.
\end{thm}
\begin{proof}
The proof of \cite[Proposition A.2.8.2]{LurieHTT} proves the existence of the projective model structure on $\C^\E$ when $\C$ is cofibrantly generated: the hypothesis of combinatoriality is not used in this part of the proof. But in fact a careful reading shows that exactly the same proof works in our situation, and shows that $\C^\E$ is class-cofibrantly generated. Then the final statement follows exactly as in \cite[Proposition A.2.8.7(1)]{LurieHTT}.
\end{proof}

We can now define homotopy colimits (of shape $\E$) as the left derived functor $\hocolim: \C^\E \to \C$ of $\colim$.

\begin{rem}
The question of the existence of homotopy limits is a more delicate one here; the usual requirement for the existence of an injective model structure is that $\C$ be combinatorial. It may be possible to pursue a definition of class-combinatorial, analogously to the definition of class-cofibrantly generated, and construct homotopy limits in that way. But we will not need this here.
\end{rem}

Let us now define the model structures we will need in this paper. First on $\U$: a weak equivalence (respectively, fibration) is a map $X \to Y$ such that the induced map $\underline{\U}(K,X) \to \underline{\U}(K,Y)$ is a weak homotopy equivalence (respectively, Serre fibration) for all compact, Hausdorff spaces $K$. A cofibration is a retract of a composition of pushouts by maps of the form $K \times D^n \times \{0\} \to K \times D^n \times [0,1]$, where $D^n$ is the $n$-ball.

Next we consider the category of simplicial spaces $s\U = \U^{\Delta^{op}}$. Here $\Delta$ is the simplex category -- see \cite[Chapter 3]{Hovey} for details on $\Delta$ and simplicial sets. This category is also cartesian closed by \cite[Proposition 3.2]{Myself}, so in particular it is enriched, via a forgetful functor, over simplicial sets $sSet$ and over $\U$. We get for free a projective model structure induced from the compact Hausdorff model structure on $\U$, but here we consider a different one. See \cite{Myself} for more details.

Let $\Lambda^n_k$, $\partial \Delta^n$ and $\Delta^n$ be the standard simplicial $(n,k)$-horn, $n$-sphere and $n$-simplex, respectively. Let $I$ be the class of maps in $s\U$ of the form $$\iota \times \id_K: \partial \Delta^n \times K \to \Delta^n \times K,$$ $K$ compact Hausdorff, and let $J$ be the class of maps of the form $$\iota' \times \id_K: \Lambda^n_k \times K \to \Delta^n \times K,$$ $K$ compact Hausdorff, where $\iota: \Lambda^n_k \to \Delta^n, \iota': \partial \Delta^n \to \Delta^n$ are the inclusion maps. Then a map $f$ in $s\U$ is a cofibration if it is a retract of a composition of pushouts by maps in $I$, a fibration if it has the right lifting property with respect to $J$, and a weak equivalence if $\underline{s\U}_{sSet}(\disc K,f)$ is a weak equivalence in $sSet$ for all compact Hausdorff $K$. Here $\disc$ is the constant simplicial space functor.

Finally, we also consider the categories of simplicial objects in $R$-$\U Mod$ and chain complexes in non-negative degrees in $R$-$\U Mod$. We write $s(R$-$\U Mod)$ for the former category and $c(R$-$\U Mod)$ for the latter. By \cite[Theorem 1.2.3.7]{LurieHA}, these two categories are equivalent, and correspondingly we get equivalent model structures on them. We give an explicit description of the structure on $c(R$-$\U Mod)$: a map is a weak equivalence if its mapping cone is exact in the compact Hausdorff exact structure, a cofibration if it is a levelwise split monomorphism, and a fibration if it is a levelwise deflation in the compact Hausdorff exact structure.

\begin{thm}
These data define model structures on $\U$, $s\U$ and $c(R$-$\U Mod)$, which we call the compact Hausdorff model structures. All three are class-cofibrantly generated.
\end{thm}
\begin{proof}
\cite[Theorem 2.4, Theorem 4.19, Theorem 8.2]{Myself}
\end{proof}

These categories have various Quillen functors between them, as shown in the following diagram: $$\U \leftrightarrows_{\underline{Sing}}^{\lvert - \rvert} s\U \rightleftarrows^{F}_{U} s(R\text{-}\U Mod) \cong c(R\text{-}\U Mod).$$ Here $\Sing$ is the singular simplicial space functor, $\lvert - \rvert$ is geometric realisation, $F$ is the free $R$-module functor and $U$ is the forgetful functor; $\lvert - \rvert$ is left adjoint to $\underline{Sing}$ and $F$ is left adjoint to $U$. See \cite{Myself} for details.

We also need the standard Quillen model structures on $\U$ and $sSet$ as defined in \cite{Hovey}, where the weak equivalences are weak homotopy equivalences.

Finally, we need one more homotopical structure one $s\U$, which we call the regular structure. In \cite{Myself} the regular structure is referred to as $(s\U,reg)$. The regular structure is defined on internal Kan complexes in $s\U$, that is, objects $X$ such that $X_n \to \Lambda^n_k(X)$ is a regular epimorphism for all $n,k$. On these objects $X$ we can define homotopy group objects $\pi_n(X,x)$ in the exact completion of $\U$, for a choice of basepoint $x$. Here we say a map $X \to Y$ is a weak equivalence if it induces isomorphisms of all homotopy groups, a fibration if $X_n \to \Lambda^n_k(X) \times_{\Lambda^n_k(Y)} Y_n$ is a regular epimorphism for all $n,k$ and a trivial fibration if $X_n \to \partial\Delta^n(X) \times_{\partial\Delta^n(Y)} Y_n$ is a regular epimorphism for all $n$.

We would like to extend some of this structure to all of $s\U$. Certainly the definition of fibrations and trivial fibrations makes sense. As for abstract simplicial sets, there is a functor $\Ex^\infty$ (see \cite{Low}).

It is stated in \cite[Theorem 2.8]{Low} that in a regular category, if finite limits commute with colimits of sequences indexed by $\mathbb{N}$, then $\Ex^\infty$ preserves fibrations and trivial fibrations, $\Ex^\infty(X)$ is an internal Kan complex for all $X$ in $s\U$, and if $X$ is an internal Kan complex, the canonical map $X \to \Ex^\infty(X)$ is a weak equivalence.

This does not hold in $\U$. But we do have the following.

\begin{lem}
\label{finitelim}
Suppose $D$ is a finite category and we have a sequence of functors $(F_n: D \to \U)_{n \in \mathbb{N}}$ and natural transformations $F_n \to F_{n+1}$. Suppose that for all $d \in D$ and all $n$ the maps $F_n(d) \to F_{n+1}(d)$ are closed inclusions. Then the induced maps $\lim F_n \to \lim F_{n+1}$ are closed inclusions and $\colim_n \lim_D F_n(D) = \lim_D \colim_n F_n(D)$.
\end{lem}
\begin{proof}
The induced maps of limits are closed inclusions because closed inclusions are equalisers in $\U$ and limits commute.

When $D$ is a pullback diagram $A \to C \leftarrow B$ and all the maps $F_n(B) \to F_{n+1}(B)$ and $F_n(C) \to F_{n+1}(C)$ are identities, the commuting of limit and colimit is a special case of \cite[Corollary 10.4]{Lewis}. The general case where $D$ is a pullback diagram follows easily from this. Finite products commute with all colimits because $\U$ is cartesian closed.

When $D$ is an equaliser diagram $f,g: A \to B$, the equaliser of $f,g: F_n(A) \to F_n(B)$ is the pullback of $$F_n(A) \xrightarrow{(f,g)} F_n(B) \times F_n(B) \xleftarrow{(\id_{F_n(B)},\id_{F_n(B)})} F_n(B).$$ So
\begin{align*}
\colim_n \lim_D F_n(D) &= \colim_n \lim_D F_n(A) \to F_n(B) \times F_n(B) \leftarrow F_n(B) \\
&= \lim_D \colim_n F_n(A) \to F_n(B) \times F_n(B) \leftarrow F_n(B) \\
&= \lim_D \colim_n F_n(D).
\end{align*}

All finite limits can be constructed as equalisers of maps between finite products, so a similar argument shows that the lemma holds for these too.
\end{proof}

All the colimits indexed by $\mathbb{N}$ in the argument for \cite[Theorem 2.8]{Low} are colimits of closed inclusions, because $X_n \to \Ex(X)_n$ is a split monomorphism, and hence a regular monomorphism. So its conclusions hold here too, and we get:

\begin{prop}
$\Ex^\infty$ preserves fibrations and trivial fibrations, $\Ex^\infty(X)$ is an internal Kan complex for all $X$ in $s\U$, and if $X$ is an internal Kan complex, the canonical map $X \to \Ex^\infty(X)$ is a weak equivalence.
\end{prop}

Then we can define regular weak equivalences in $s\U$ to be maps $X \to Y$ such that $\Ex^\infty(X) \to \Ex^\infty(Y)$ is a regular weak equivalence. It follows easily that weak equivalences in $s\U$ in the compact Hausdorff structure are weak equivalences in the regular structure: see \cite[Lemma 5.16]{Myself}.

Just as we can define homotopy group objects for $X \in \U$, we can now define homology group objects too. To ensure that these are invariant under weak equivalence in the compact-Hausdorff structure, we define $H_n(X)$ to be $LH_n \circ LF \circ \underline{\Sing}(X)$ -- where $LF$ is the left derived functor of the free group functor $F$, calculated in the compact Hausdorff structure. This works because weak equivalences in $c(\U Ab)$ in the compact Hausdorff structure are weak equivalences in the regular structure.

\section{\texorpdfstring{$\underline{\Sing}$ and $\lvert - \rvert$}{Sing and geometric realisation}}
\label{Sing}

Write $\Sing$ for the usual singular simplicial set functor. As for $\underline{\Sing}$, it has a geometric realisation functor $\lvert - \rvert: sSet \to \U$ as a left adjoint. It is a standard result of homotopical algebra that the two compositions $\Sing \lvert - \rvert$ and $\lvert \Sing(-) \rvert$ are weakly equivalent to the identity. In fact the same is true of $L\lvert \underline{\Sing}(-) \rvert$, by \cite[Proposition 4.29]{Myself}. But it is easy to find counter-examples showing the same is not true of the other composition. Much of the rest of the paper will go into showing a partial result in this direction, Theorem \ref{tpd}. In fact the construction of Eilenberg--Mac Lane spaces for totally path-disconnected groups follows quite easily from this result.

We can now state the main technical result of the paper. This can be thought of as a continuous version of the Seifert--van Kampen Theorem.

\begin{thm}
\label{singhoco}
If $C$ is an open cover of $X \in \U$, write $C'$ for the poset of finite intersections of sets in $C$, ordered by inclusion. Then $\underline{\Sing}(X)$ is weakly equivalent (in the regular structure on $s\U$) to the homotopy colimit (in the compact Hausdorff structure on $s\U$) of $\{\underline{\Sing}(U)\}_{U \in C'}$.
\end{thm}

\begin{rem}
This mixture of the two homotopical structures is not beautiful, but it is a necessary evil. We will see later in Example \ref{pseudo} that in general the homotopy colimit of $\{\underline{\Sing}(U)\}_{U \in C'}$ is not weakly equivalent to $\underline{\Sing}(X)$ in the compact Hausdorff structure; on the other hand, without a well-behaved notion of cofibrations in the regular structure, we have no way of defining homotopy colimits there -- but see Section \ref{excision}.
\end{rem}

This may be thought of as a `continuous version' of \cite[Proposition A.3.2]{LurieHA}, and we will start by reproving the intermediate result \cite[Lemma A.3.3]{LurieHA}, using an approach that carries across better to the current situation.

\begin{lem}
(\cite[Lemma A.3.3]{LurieHA}) Let $X \in \U$, and let $C$ be an open cover of $X$. Let $\Sing'(X)$ be the simplicial subset of $\Sing(X)$ spanned by those $n$-simplices $\lvert \Delta^n \rvert \to X$ which factor through some $U \in C$. Then the inclusion $i : \Sing'(X) \to \Sing(X)$ is a weak equivalence of simplicial sets.
\end{lem}
\begin{proof}
Via the adjunction between $\Sing$ and $\lvert - \rvert$, we see that $\Ex(\Sing'(X))_n \cong s\U(\sd \Delta^n, \Sing'(X))$ is naturally isomorphic to the subset of $\U(\lvert \sd \Delta^n \rvert, X)$ which maps each simplex of $\sd \Delta^n$ into some $U \in \C$. In this way we think of $\Ex(\Sing'(X))$ as a simplicial subset of $\Sing(X)$. Iterating this reasoning, we can identify the limit $\Ex^\infty(\Sing'(X))$ as the simplicial subset of $\Sing(X)$ consisting of all maps $\lvert \Delta^n \rvert \to X$ for which there is some $m \in \mathbb{N}$ such that each simplex of $\sd^m \Delta^n$ is mapped into some $U \in \C$. By standard arguments, e.g. \cite[Proof of Proposition 2.21]{Hatcher}, this is all of them. That is, $\Ex^\infty(\Sing'(X)) = \Sing(X)$.

There is a subtlety here. There is a canonical map $e: S \to \Ex S$ for a simplicial set $S$, but that is not the inclusion map we are using here. It is well-known that $e$ is a trivial cofibration, and hence $e^\infty: S \to \Ex^\infty S$ is too -- we want to show our inclusion maps are all trivial cofibrations, and the result will follow. Since all our inclusion maps are injective and hence cofibrations of simplicial sets, it suffices to show each inclusion, which we will write (by abuse of notation) as $$i: \Ex^m(\Sing'(X)) \to \Ex^{m+1}(\Sing'(X)),$$ is a weak equivalence. We do this by showing $i$ is homotopic to the weak equivalence $$e: \Ex^m(\Sing'(X)) \to \Ex^{m+1}(\Sing'(X)),$$ after which the result follows by standard model category theory.

This homotopy comes simply from understanding what the maps $i$ and $e$ are doing: $i$ is induced by the identity map $\lvert \sd^m \Delta^n \rvert \to \lvert \Delta^n \rvert$, while $e$ is induced by the map which identifies one of the $n$-simplices of $\lvert \sd^m \Delta^n \rvert$ with $\lvert \Delta^n \rvert$, and retracts all the other simplices onto faces of $\lvert \Delta^n \rvert$. These two maps are clearly homotopic, and a choice of homotopy induces a homotopy between $i$ and $e$.
\end{proof}

We can immediately prove a continuous version:

\begin{lem}
\label{colimwe}
Let $X \in \U$, and let $C$ be an open cover of $X$. Let $\underline{\Sing}'(X)$ be the simplicial subset of $\underline{\Sing}(X)$ spanned by those $n$-simplices $\lvert \Delta^n \rvert \to X$ which factor through some $U \in C$. Then the inclusion $i : \underline{\Sing}'(X) \to \underline{\Sing}(X)$ is a weak equivalence of simplicial spaces in the compact Hausdorff structure.
\end{lem}
\begin{proof}
It is an easy exercise to check that each $\Ex^m(\underline{\Sing}'(X))$ has $\Ex^m(\Sing'(X))$ as its underlying set, and that its topology is the subspace topology from its inclusion map into $\underline{\Sing}(X)$, defined as in the previous lemma. Hence $\Ex^\infty(\underline{\Sing}'(X)) = \underline{\Sing}(X)$ as before and we have to show the resulting map $\underline{\Sing}'(X) \to \Ex^\infty(\underline{\Sing}'(X))$ is a weak equivalence. This works in exactly the same way as in \cite[Theorem 4.19 (ii)]{Myself}: first, for any $X \in \U$ and $Y \in s\U$, $\underline{s\U}_{sSet}(\disc X,\Ex(Y)) = \Ex(\underline{s\U}_{sSet}(\disc X,Y)$:
\begin{align*}
\underline{s\U}_{sSet}(\disc X,\Ex(Y))_n &= \underline{\U}(X,(\Ex(Y)_n) \\
&= \{\sd \Delta^n,\underline{s\U}_{sSet}(\disc X,Y)\} \\
&= \Ex(\underline{s\U}_{sSet}(\disc X,Y)_n.
\end{align*}
Note that each $\Ex^m(\underline{\Sing}'(X))_n$ is open in $\underline{\U}(\Delta^n,X)$ via the inclusion $i$, by definition of the compact-open topology: it is the subset consisting of the finite intersection (over simplices $K$ of $\lvert \sd^m \Delta^n \rvert$) of the unions (over open sets $U$ in the cover $C$) of the open sets $O(K,U)$. So the maps $$i_n: \Ex^m(\underline{\Sing}'(X))_n \to \Ex^{m+1}(\underline{\Sing}'(X))_n$$ are open inclusions. Since every compact Hausdorff space $K$ is small with respect to open inclusions (that is, any compact Hausdorff subspace of $\underline{\Sing}(X)_n$ must be contained in one of the sequence $\Ex^m(\underline{\Sing}'(X))_n$ of open subspaces), we get $$\underline{s\U}_{sSet}(\disc K,\Ex^\infty(\underline{\Sing}'(X))) = \Ex^\infty(\underline{s\U}_{sSet}(\disc K,\underline{\Sing}'(X)))$$ for all $K$. By the previous lemma, $$i: \underline{s\U}_{sSet}(\disc K,\underline{\Sing}'(X)) \to \Ex^\infty(\underline{s\U}_{sSet}(\disc K,\underline{\Sing}'(X)))$$ is a weak equivalence for all $K$, so the result follows.
\end{proof}

We want a concrete model for the homotopy colimit of $\{\underline{\Sing}(U)\}_{U \in C'}$ (in the compact Hausdorff structure). The model structure on $s\U$ gives a cofibrant replacement functor $Q$: explicitly, for $Y$ in $s\U$, $Q(Y)_0$ is the disjoint union of the compact subspaces of $Y_0$ and $Q(Y)_n$ is the disjoint union of the compact subspaces of the pullback $\partial\Delta^n(Q(Y)) \times_{\partial\Delta^n(Y)} Y_n$. This gives each $Q(Y)_n$ a canonical decomposition as a disjoint union of compact subspaces which we want to fix for later: call these compact subspaces `fat cells'. The intuition is that fat $n$-cells should be treated like single $n$-cells of simplicial sets which have been fattened up.

\begin{lem}
\label{hocolim}
$\colim_{U \in C'} Q(\underline{\Sing}(U))$ is a homotopy colimit for $\{\underline{\Sing}(U)\}_{U \in C'}$.
\end{lem}
\begin{proof}
Let $C''$ be the poset of all intersections of open sets in $C$. We will show $\{Q(\underline{\Sing}(U))\}_{U \in C''}$ is cofibrant in the projective model structure induced from the compact Hausdorff model structure on $s\U$; the colimit of $\{Q(\underline{\Sing}(U))\}_{U \in C''}$ is the colimit of $\{Q(\underline{\Sing}(U))\}_{U \in C'}$ by cofinality, and the result follows. The argument echoes the second part of the proof of \cite[Proposition A.3.2]{LurieHA}.

Write $Q(\underline{\Sing}'(X))$ as a transfinite colimit $(A_\alpha)$ of pushouts by maps of the form $\partial\Delta^n \times K_\alpha \to \Delta^n \times K_\alpha$ with $K_\alpha$ compact Hausdorff (that is, the generating cofibrations of the model structure), ordered by dimension. Identify each $Q(\underline{\Sing}(U))$ with its image as a simplicial subspace of $Q(\underline{\Sing}'(X))$, consisting of a subset of the fat cells. Then the result follows by showing that for each $A_\alpha \to A_{\alpha+1}$ the induced map $$\{Q(\underline{\Sing}(U) \cap A_\alpha)\}_{U \in C''} \to \{Q(\underline{\Sing}(U) \cap A_{\alpha+1})\}_{U \in C''}$$ is a cofibration in the projective model structure, since these are closed under transfinite composition. Let $V \in C''$ be the intersection of sets in $C$ which contain the image of $\Delta^n \times K_\alpha$: then this map is a pushout by the projective cofibration $F_0 \to F$, where
\[
F_0(U) =
\begin{cases}
K_\alpha \times \partial\Delta^n & \text{if } V \subseteq U,\\
\emptyset & \text{otherwise,}
\end{cases}
F(U) =
\begin{cases}
K_\alpha \times \Delta^n & \text{if } V \subseteq U,\\
\emptyset & \text{otherwise.}
\end{cases}
\]
\end{proof}

It remains to show that $\colim_{U \in C'} Q(\underline{\Sing}(U))$ is weakly equivalent in the regular structure to $\underline{\Sing}(X)$. But first we will give an example to show that it is not a weak equivalence in general in the compact Hausdorff structure.

Factor the canonical map $\colim_{U \in C'} Q(\underline{\Sing}(U)) \to \underline{\Sing}(X)$ using the functorial factorisation into a trivial cofibration followed by a fibration $$\colim_{U \in C'} Q(\underline{\Sing}(U)) \to Z \to \underline{\Sing}(X).$$

\begin{example}
\label{pseudo}
Let $X$ be the pseudo-arc, as defined in \cite{Bing}. This is a compact Hausdorff space (so in $\U$) which is connected and totally path-disconnected. Let $C$ be any open cover of $X$ which does not contain $X$ itself.

To show the fibration $Z \to \underline{\Sing}(X)$ is not a trivial fibration it is enough to show that $Z_0 \to \underline{\Sing}(X)_0$ is not $CH$-split, in the terminology of \cite{Myself}: that is, that not every map from a compact Hausdorff space to $\underline{\Sing}(X)_0$ lifts to a map to $Z_0$. More specifically, since $\underline{\Sing}(X)_0 = X$ is itself compact Hausdorff, it is enough to show $Z_0 \to X$ does not split.

Indeed, in the functorial factorisation, $Z$ is the colimit of a sequence $Z^n \in s\U$ where new $0$-cells in $Z^n$ are attached via spaces of $1$-cells in $\underline{\Sing}(X)$ (or equivalently via spaces of paths in $X$) with one end in $Z^{n-1}$, and $Z^0 = \colim_{U \in C'} Q(\underline{\Sing}(U))$. Since all paths in $X$ are constant, and all spaces of $0$-cells in $Z^0$ are disjoint unions of spaces whose image is contained in some $U \in C$, the same is true of $Z$. So $Z_0$ is a disjoint union of spaces whose image is not the whole of $X$. Therefore any splitting $X \to Z_0$ would disconnect $X$, giving a contradiction.
\end{example}

We will prove the theorem by showing that $Z \to \underline{\Sing}(X)$ is a trivial fibration in the regular structure. To do this we must first understand $Z$ better. We set $Z^0 = \colim_{U \in C'} Q(\underline{\Sing}(U))$. Given $Z^{m-1}$, for all $n \in \mathbb{N}$, for all $0 \leq k \leq n$, for every compact subspace $K$ of $\Lambda^n_k(Z^{m-1}) \times_{\Lambda^n_k(\underline{\Sing}(X))} \underline{\Sing}(X)_n$, we attach $K$ $(m-1)$-cells and $K$ $m$-cells filling in this space of $n$-horns, and call the resulting space $Z^m$. Then $Z$ is the colimit of the sequence $Z^0 \to Z^1 \to \cdots$ with the obvious inclusion maps.

The strategy will be to show that the map $$Z_n \to \partial\Delta^n(Z) \times_{\partial\Delta^n(\underline{\Sing}(X))} \underline{\Sing}(X)_n$$ is a $\beta$-epimorphism for all $n$, in the terminology of \cite{Strickland}: that is, that for every compact subspace of the pullback there is a compact subspace of $Z_n$ mapping onto it. The proposition will follow since $\beta$-epimorphisms are regular epimorphisms by \cite[Proposition 3.12]{Strickland}.

We do this by proving a slightly stronger result. Instead of just considering the standard $n$-simplex $\Delta^n$, we wish to consider every finite triangulation $\Sigma$ of the $n$-ball. Just as for $\Delta^n$, we may consider its boundary $\partial\Sigma$, which is a finite triangulation of the $n$-sphere. Just as we may speak of spaces of $n$-boundaries $\partial\Delta^n(Z)$ in a simplicial space $Z$, we may define $\partial\Sigma(Z) = \{\partial\Sigma,Z\}$, the weighted limit of $Z$ over $\partial\Sigma$; for more detail on weighted limits see \cite{Riehl}.

Theorem \ref{singhoco} will follow once we prove:

\begin{prop}
\label{beta}
For all such $\Sigma$, the map $$\{\Sigma,Z\} \to \{\partial\Sigma,Z\} \times_{\{\partial\Sigma,\underline{\Sing}(X)\}} \{\Sigma,\underline{\Sing}(X)\}$$ is a $\beta$-epimorphism.
\end{prop}

We prove this in several steps.

By Lemma \ref{finitelim}, for each finite triangulation $\Sigma$, $\{\Sigma,Z\}$ is a finite limit, so it is the union of the closed subspaces $\{\Sigma,Z^m\}$, so any compact subspace of $\{\Sigma,Z\}$ is contained in some $\{\Sigma,Z^m\}$. Similarly using $\partial\Sigma$. This allows us to use an inductive argument; we start with the base case.

Suppose we have a compact subspace $K$ of $\{\partial\Sigma,Z\} \times_{\{\partial\Sigma,\underline{\Sing}(X)\}} \{\Sigma,\underline{\Sing}(X)\}$. Label the canonical maps from the pullback $$\{\partial\Sigma,Z\} \xleftarrow{f} \{\partial\Sigma,Z\} \times_{\{\partial\Sigma,\underline{\Sing}(X)\}} \{\Sigma,\underline{\Sing}(X)\} \xrightarrow{g} \{\Sigma,\underline{\Sing}(X)\}.$$ We can assume that each simplex of $f(K)$ (that is, the image of $f(K)$ under the map $\{\partial\Sigma,Z\} \to \{\Delta^m,Z\}$ induced by a simplex inclusion $\Delta^m \to \partial\Sigma$) is contained in a single fat cell, since $K$ is the disjoint union of finitely many spaces of this form.

\begin{prop}
If $f(K)$ is contained in the subspace $\{\partial\Sigma,Z^0\}$, there is a compact cover $K_1, \ldots, K_j$ of $K$ such that the inclusion map of each $K_i$ into $\{\partial\Sigma,Z\} \times_{\{\partial\Sigma,\underline{\Sing}(X)\}} \{\Sigma,\underline{\Sing}(X)\}$ lifts to a map $K_i \to \{\Sigma,Z\}$.
\end{prop}
\begin{proof}
We will start by replacing $K$ with an associated subspace $A(K)$ of $\{\partial\Sigma,Z\} \times_{\{\partial\Sigma,\underline{\Sing}(X)\}} \{\Sigma,\underline{\Sing}(X)\}$. For each $m$-cell $D$ of $\Sigma$, $K$ has a space $K_D$ of $m$-cells in $Z^0 = \colim_{C'} Q(\underline{\Sing}(U)$, which by hypothesis is contained in a single fat cell, that is, some compact subspace of some $U \in C$. The image of $K_D$ in $\underline{\Sing}(U)$ is compact, so it corresponds to some fat cell $K'_D$ in $Z^0$. Replacing $K_D$ with $K'_D$ is functorial in the cells of $\partial\Sigma$, so together these spaces of cells give a subspace $K'$ of $\{\partial\Sigma,Z\}$, homeomorphic to $f(K)$ whose image in $\{\partial\Sigma,\underline{\Sing}(X)\}$ is the same. Then let $A(K)$ be the pullback of $K' \to \{\partial\Sigma,\underline{\Sing}(X)\} \leftarrow g(K)$. This is canonically homeomorphic to $K$, and it is not hard to see that there is a series of simplicial homotopies in $Z^0$ between $A(K)$ and $K$.

Now consider $g(A(K)) = g(K)$. We use the same compactness argument as Lemma \ref{colimwe}: applying barycentric subdivision $k$ times to $\Sigma$, each $\{\sd^k \Sigma, \underline{\Sing}'(X)\}$ is an open subspace of $\{\Sigma, \underline{\Sing}(X)\}$ and $\bigcup_k \{\sd^k \Sigma, \underline{\Sing}'(X)\} = \{\Sigma, \underline{\Sing}(X)\}$. Since $g(K)$ is compact there must be some $\{\sd^k \Sigma, \underline{\Sing}'(X)\}$ containing it. In the compact-open topology, the space $\{\sd^k \Sigma, \underline{\Sing}'(X)\} = \{$maps from $\sd^k \Sigma$ to $X$ such that each simplex maps into some $U \in C\}$ has an open cover by sets of the form $\bigcap_{D \in \sd^k \Sigma}\bigcup_{U \in C}O(D,U)$. This gives an open cover of $g(K)$ which pulls back to an open cover of $A(K)$; pick a finite subcover $V_1, \ldots, V_j$. By the Shrinking Lemma of \cite[Exercise 4.36.4]{Munkres}, we can find a cover of $A(K)$ by compact subspaces $K_1, \ldots, K_j$ such that $K_i \subseteq V_i$ for each $i$. For each simplex $D$ in $\sd^k \Sigma$, there is some $U \in C$ such that every element of $g(K_i)$ maps $D$ into $U$. Therefore the inclusion $K_i \to \{\partial\sd^k\Sigma,Z^0\} \times_{\{\partial\sd^k\Sigma,\underline{\Sing}(X)\}} \{\sd^k\Sigma,\underline{\Sing}(X)\}$ lifts to a map $K_i \to \{\sd^k\Sigma,Z^0\}$. Because $Z \to \underline{\Sing}(X)$ is a fibration, we can stick all these spaces of cells together to get a lift $K_i \to \{\Sigma,Z\}$.

Because $Z \to \underline{\Sing(X)}$ is a fibration in the compact Hausdorff structure, we may now use the series of simplicial homotopies between $A(K)$ and $K$ to get a compact cover $K'_1, \ldots, K'_j$ of $K$ whose inclusion maps into $\{\partial\Sigma,Z\} \times_{\{\partial\Sigma,\underline{\Sing}(X)\}} \{\Sigma,\underline{\Sing}(X)\}$ lift to maps $K'_i \to \{\Sigma,Z\}$.
\end{proof}

\begin{proof}[Proof of Proposition \ref{beta}]
Now suppose we have a compact subspace $K$ of $\{\partial\Sigma,Z\} \times_{\{\partial\Sigma,\underline{\Sing}(X)\}} \{\Sigma,\underline{\Sing}(X)\}$: $f(K)$ is contained in some $\{\partial\Sigma,Z^m\}$, $m>0$. Suppose we have shown, for every compact subspace $K'$ of $\{\partial\Sigma,Z\} \times_{\{\partial\Sigma,\underline{\Sing}(X)\}} \{\Sigma,\underline{\Sing}(X)\}$ such that $f(K')$ is contained in $\{\partial\Sigma,Z^{m-1}\}$, that there is a compact cover $K'_1, \ldots, K'_j$ of $K'$ such that the inclusion map of each $K'_i$ into the pullback lifts to a map $K'_i \to \{\Sigma,Z\}$.

Suppose $\Sigma$ is a triangulation of the $(n+1)$-ball. In the process of constructing $Z$, there are two possibilities: a fat $k$-cell in $Z^m$ can be added by a $(k+1)$-horn and have all its faces in $Z^{m-1}$, or it can be added by a $k$-horn and have one face in $Z^m \setminus Z^{m-1}$, and all other faces in $Z^{m-1}$. Therefore $f(K)$ can only have fat $n$-cells and fat $(n-1)$-cells in $Z^m \setminus Z^{m-1}$; moreover, any of its fat $(n-1)$-cells in $Z^m \setminus Z^{m-1}$ is only the face of one fat $n$-cell in $Z^m$, but it is the face of two fat $n$-cells in $f(K)$, so these two fat cells must be the same. So there are two situations we will deal with:
\begin{enumerate}[(i)]
\item a fat $n$-cell in $Z^m \setminus Z^{m-1}$ with all its faces in $Z^{m-1}$;
\item two copies of the same fat $n$-cell in $Z^m \setminus Z^{m-1}$ stuck together at one face in $Z^m \setminus Z^{m-1}$.
\end{enumerate}

\begin{enumerate}[(i)]
\item For each fat $n$-cell $D$ in case (i), we have a fat $(n+1)$-cell $D'$ that has $D$ as one face and all other faces in $Z^{m-1}$. Write $\Sigma'$ for the triangulation of the $(n+1)$-ball obtained by attaching a new $(n+1)$-cell to $\Sigma$ at the face $\sigma$ corresponding to $D$. Consider the new compact subspace $K'$ of $\{\partial\Sigma',Z\} \times_{\{\partial\Sigma',\underline{\Sing}(X)\}} \{\Sigma',\underline{\Sing}(X)\}$: for each element of $K$, change its image in $\{\Sigma,\underline{\Sing}(X)\}$ by attaching a new $(n+1)$-cell via the map attaching $D'$ at $D$, and change its image in $\{\partial\Sigma,Z\}$ by replacing the image of $\sigma$ with an $(n+1)$-horn via the map attaching $D'$ at $D$. Note that $f(K')$ has fewer fat $n$-cells in $Z^m \setminus Z^{m-1}$ than $f(K)$. Note too that, if we can find a compact cover $K'_1, \ldots, K'_j$ of $K'$ such that the inclusion map of each $K'_i$ into the pullback lifts to a map $K'_i \to \{\Sigma',Z\}$, we can use the fact that $Z \to \underline{\Sing}(X)$ is a fibration in the compact Hausdorff structure to get a compact cover $K_1, \ldots, K_j$ of $K$ such that the inclusion map of each $K_i$ into $\{\partial\Sigma,Z\} \times_{\{\partial\Sigma,\underline{\Sing}(X)\}} \{\Sigma,\underline{\Sing}(X)\}$ lifts to a map $K_i \to \{\Sigma,Z\}$.

So by applying this procedure finitely many times we reduce to the case where $f(K)$ has no fat $n$-cells of this type in $Z^m \setminus Z^{m-1}$.

\item For each fat $n$-cell $D$ in case (ii), the approach is similar. Let $\Sigma'$ be the triangulation given by attaching a new $(n+1)$-cell to $\Sigma$ at the two faces corresponding to $D$. Consider the new compact subspace $K'$ of $\{\partial\Sigma',Z\} \times_{\{\partial\Sigma',\underline{\Sing}(X)\}} \{\Sigma',\underline{\Sing}(X)\}$: for each element of $K$, change its image in $\{\Sigma,\underline{\Sing}(X)\}$ by attaching a new degenerate $(n+1)$-cell coming from the degeneracy maps on $D$ to the faces corresponding to $\sigma_1$ and $\sigma_2$, and change its image in $\{\partial\Sigma,Z\}$ by replacing the images of $\sigma_1$ and $\sigma_2$ with the other faces of the new $(n+1)$-cell. These other faces are the degeneracies of $(n-2)$-cells in $f(K)$, so they are in $Z^{m-1}$; so, as before, we reduce the number of fat cells in $Z^m \setminus Z^{m-1}$. As before, if we can find a compact cover $K'_1, \ldots, K'_j$ of $K'$ such that the inclusion map of each $K'_i$ into the pullback lifts to a map $K'_i \to \{\Sigma',Z\}$, we can use the fact that $Z \to \underline{\Sing}(X)$ is a fibration in the compact Hausdorff structure to get a compact cover $K_1, \ldots, K_j$ of $K$ such that the inclusion map of each $K_i$ into $\{\partial\Sigma,Z\} \times_{\{\partial\Sigma,\underline{\Sing}(X)\}} \{\Sigma,\underline{\Sing}(X)\}$ lifts to a map $K_i \to \{\Sigma,Z\}$.

So by applying this procedure finitely many times we reduce to the case where $f(K)$ has no fat $n$-cells in $Z^m \setminus Z^{m-1}$, and we are done.
\end{enumerate}
\end{proof}

\section{Excision}
\label{excision}

If the weak equivalence in the regular structure proved in Theorem \ref{singhoco} were a weak equivalence the CH structure, an excision theorem for homology would be an easy corollary: since left derived functors preserve homotopy colimits, we could apply $LF$ to get a result analogous to \cite[Proposition 2.21]{Hatcher}, and deduce excision from there exactly as in \cite{Hatcher}. As it is not, we have to work a bit harder.

As before, let $C$ be an open cover of $X \in \U$, and write $C'$ for the poset of finite intersections of sets in $C$, ordered by inclusion. The idea here is that the map $\colim_{U \in C'} Q(\underline{\Sing}(U)) \to Q(\underline{\Sing}'(X))$ looks a lot like a trivial cofibration, even though it is not one in general. We can produce $Q(\underline{\Sing}'(X))$ from $\colim Q(\underline{\Sing}(U))$ by attaching collections of cells which are close enough to trivial cofibrations that they behave nicely under $F$.

Let $Y^0 = \colim Q(\underline{\Sing}(U))$. We will create $Y^\alpha$, for ordinals $\alpha$, by adding spaces of cells to $Y^{\alpha-1}$ when $\alpha$ is a successor, and by taking $Y^\alpha = \colim_{\beta < \alpha} Y^\beta$ for $\alpha$ a limit. At each step $Y^\alpha$ will be a simplicial subspace of $Q(\underline{\Sing}'(X))$ consisting of a subset of the fat cells.

We need two new operations: filling empty horns and gluing fat cells together.

An empty $n$-horn in $Y^\alpha$ is a fat $n$-cell $K$ of $Q(\underline{\Sing}'(X))$ which is not in $Y^\alpha$, such that all but one of its face maps $K \to Q(\underline{\Sing}'(X))_{n-1}$ factors through $Y^\alpha$ but there is some $k$ such that its $k$th face map does not. We can attach $K$ $(n-1)$-cells and $K$ $n$-cells to $Y^\alpha$ via these maps, and refer to this as filling an empty horn in $Y^\alpha$.

Suppose we have a fat $n$-cell $K$ of $Q(\underline{\Sing}'(X))$ which is not in $Y^\alpha$, but all of its face maps factor through $Y^\alpha$. Suppose there is a finite cover of $K$ by compact subspaces $K_1, \ldots, K_j$, such that there are fat $n$-cells $D_i$ in $Y^\alpha$ for each $i$ and $D_i$ and $K_i$ have the same image in the pullback $\partial\Delta^n(Y^\alpha) \times_{\partial\Delta^n(Q(\underline{\Sing}'(X)))} Q(\underline{\Sing}'(X))_n$ for each $i$. We want to glue the $K_i$s together to get $K$, but it is not yet clear whether this can done without changing $Y^\alpha$ homotopically `too much', in a sense which will become clear later. For now, call this a $1$-valid opportunity.

If, for each non-empty intersection $K_i \cap K_j$, $Y^\alpha$ has a fat cell $D_{i,j}$ of $K_i \cap K_j$ $(n+1)$-cells filling the space of $(n+1)$-boundaries created by the subspaces $K_i \cap K_j$ of cells of $D_i$ and $D_j$, we say the opportunity is $2$-valid.

Similarly, we inductively say an $(r-1)$-valid opportunity is $r$-valid if, for each non-empty intersection $K_{i_1} \cap \cdots \cap K_{i_r}$, $Y^\alpha$ has a fat cell $D_{i_1, \ldots, i_r}$ of $K_{i_1} \cap \cdots \cap K_{i_r}$ $(n+r-1)$-cells filling the space of $(n+r-1)$-boundaries created by the subspaces $K_{i_1} \cap \cdots \cap K_{i_r}$ of cells of $D_{i_1, \ldots, i_{s-1}, i_{s+1}, \ldots, i_r}$ for $1 \leq s \leq r$. If an opportunity is $r$-valid for all $r$ (equivalently, if it is $j$-valid), we say it is valid, and call $n$ its dimension.

Given a valid opportunity, we can glue fat cells together. Attach $D$ to $Y^\alpha$. For each $i$, attach a fat cell $E_i$ of $K_i$ $(n+1)$-cells to the space of $(n+1)$-boundaries created by $D_i$ and the subspace $K_i$ of cells of $D$. Continue inductively: for each non-empty intersection $K_{i_1} \cap \cdots \cap K_{i_r}$, attach a fat cell $E_{i_1, \ldots, i_r}$ of $K_{i_1} \cap \cdots \cap K_{i_r}$ $(n+r)$-cells to the space of $(n+r)$-boundaries created by $D_{i_1, \ldots, i_r}$ and the subspaces $K_{i_1} \cap \cdots \cap K_{i_r}$ of cells of $E_{i_1, \ldots, i_{s-1}, i_{s+1}, \ldots, i_r}$ for $1 \leq s \leq r$. Note that this process terminates after at most $j$ steps.

Now we can define $Y^\alpha$ inductively: for $\alpha$ a successor, if $Y^{\alpha-1}$ has any empty horns, pick one, fill it, and call the result $Y^\alpha$. If it has no empty horns, but has valid opportunities to glue fat cells together, pick one of minimal dimension, take it, and call the result $Y^\alpha$. For $\alpha$ a limit we let $Y^\alpha = \colim_{\beta < \alpha} Y^\beta$.

Note that when $Y^\alpha$ is a simplicial subspace of $Q(\underline{\Sing}'(X))$ consisting of a subset of the fat cells, $Y^{\alpha+1}$ is too, so we can see inductively that this is true for all $Y^\alpha$. So this process must terminate: there is some ordinal $\gamma$ such that $Y^\gamma$ has no empty horns and no valid opportunities.

As we are building $Y^\gamma$, we can label each fat cell $K$ with a finite compact cover $l(K)$ of the space of cells. For fat cells in $Y^0$, take $l(K) = \{K\}$, and assume we have labelled all the fat cells in $Y^\beta$ for all $\beta < \alpha$. Assume that all the labels for $\beta < \alpha$ have the following two properties. First, any face map of a fat cell $K$ in $Y^\gamma$ sends any $K_i \in l(K)$ into some element of the label of the face. Second, for any fat cell $K$ in $Y^\gamma$ and $K_i \in l(K)$, the image of the subspace $K_i$ in $\underline{\Sing}'(X)$ is contained in some $\underline{\Sing}(U)$.

For $\alpha$ a limit there is no more to do; suppose it is a successor. When we create $Y^\alpha$ from $Y^{\alpha-1}$ by filling an empty $k$-horn, each face map except the $k$th on the new fat cell pulls back to a finite compact cover of $K$; pick a common refinement of all of these: that is, a finite compact cover $K_1, \ldots, K_j$ such that the image of every $K_i$ under every face map except the $k$th is contained in one of the compact spaces in the finite cover labelling that space. For each $K_i$, we get an open cover by sets of the form $\{x \in K_i:$ the image of $\{x\}$ in $K \to \underline{\Sing}'(X)$ is contained in $\underline{\Sing}(U), U \in C\}$. Pick a finite subcover of this, and then a compact refinement $K_{i,1}, \ldots, K_{i,l_i}$, which exists by the Shrinking Lemma, \cite[Exercise 4.36.4]{Munkres}. The compact cover of $K$ given by $\{K_{1,1}, \ldots, K_{1,l_1}, \ldots, K_{j,1}, \ldots, K_{j,l_j}\}$ has the required properties, showing such a cover exists. Now take a compact cover of $K$ of minimal size satisfying these properties to be the labels of $K$ and its $k$th face.

If $Y^\alpha$ comes from $Y^{\alpha-1}$ by gluing cells, we are attaching finitely many fat cells by filling spaces of boundaries. For each of these fat cells $K$, we pick some finite compact cover which is a common refinement of all the finite compact covers induced by the face maps. We refine it further, as above, to show a label satisfying the required properties exists; then we label $K$ with a label of minimal size satisfying these properties.

\begin{prop}
Consider a fat cell $K$ in $Q(\underline{\Sing}'(X))$ whose image in $\underline{\Sing}(X)$ is contained in some $\underline{\Sing}(U)$, whose image under each face map is in some $Y^\alpha$, and each such image is contained in a single element of the label of a single fat cell. Then $K$ is in some $Y^\beta$, $\beta \geq \alpha$, obtained from $Y^\alpha$ just by filling empty horns.
\end{prop}
\begin{proof}
This is proved by induction on $\alpha$; thanks to our hypotheses on the labels, everything works when we imitate the proof of Proposition \ref{beta}.
\end{proof}

\begin{prop}
$Y^\gamma = Q(\underline{\Sing}'(X))$.
\end{prop}
\begin{proof}
Suppose not. Let $n$ be the minimal dimension of a fat cell of $Q(\underline{\Sing}'(X))$ which is not in $Y^\gamma$, and let $K$ be such a fat cell. For simplicity we can assume that every simplex of $K$ (that is, the image of $K$ under any sequence of face maps) is contained in a single fat cell, since $K$ is the disjoint union of finitely many spaces of cells of this form.

By hypothesis, all the faces $d_0(K), \ldots, d_n(K)$ of $K$ are in $Y^\gamma$, and already labelled. We can label $l(K) = \{K_1, \ldots, K_j\}$ as above, so that any face map of $K$ sends any $K_i$ into some element of the label of the face, and the image of each $K_i$ in $\underline{\Sing}'(X)$ is contained in some $\underline{\Sing}(U)$. We will show that $K$ can be added to $Y^\gamma$ by gluing cells together, giving a contradiction.

For every $K_i \in K$, the corresponding fat cell $D_i$ such that $D_i$ and $K_i$ have the same image in the pullback $\partial\Delta^n(Y^\alpha) \times_{\partial\Delta^n(Q(\underline{\Sing}'(X)))} Q(\underline{\Sing}'(X))_n$ has a single element in its label, since we have chosen the labels to have minimal size. So by the previous proposition, $D_i$ is in $Y^\gamma$. Moreover, we see inductively that for every non-empty intersection $K_{i_1} \cap \cdots \cap K_{i_r}$, every $D_{i_1, \ldots, i_{s-1}, i_{s+1}, i_r}$ has a single element in its label, and the subspaces $K_{i_1} \cap \cdots \cap K_{i_r}$ of the cells of each $D_{i_1, \ldots, i_{s-1}, i_{s+1}, i_r}$ form the boundary in $Q(\underline{\Sing}'(X))$ of a space of degenerate $(n+r-1)$-cells whose image in $\underline{\Sing}'(X)$ is contained in the image of $K_{i_1} \cap \cdots \cap K_{i_r}$. By the previous proposition again, $D_{i_1, \ldots, i_r}$ is in $Y^\gamma$. Since this holds for all non-empty $K_{i_1} \cap \cdots \cap K_{i_r}$, we have a valid opportunity to add $K$ to $Y^\gamma$ by gluing fat cells, giving a contradiction.
\end{proof}

\begin{lem}
\label{colimexact}
Colimits of sequences of closed inclusions are exact in $\U Ab$ in the regular structure. That is, given short exact sequences $A_n \rightarrowtail B_n \twoheadrightarrow C_n$ in the regular structure on $\U Ab$, and commutative diagrams
\[
\xymatrix{A_n \ar@{>->}[r] \ar@{>->}[d] & B_n \ar@{->>}[r] \ar@{>->}[d] & C_n \ar@{>->}[d] \\
A_{n+1} \ar@{>->}[r] & B_{n+1} \ar@{->>}[r] & C_{n+1},}
\]
the induced sequence $\colim A_n \to \colim B_n \colim C_n$ is exact.
\end{lem}
\begin{proof}
Given any diagram of groups $\{G_i\}$ in $\U Ab$, the colimit in $Ab$ with the colimit topology (if this is in $\U$) becomes an object in $\U Ab$. Indeed, the only non-trivial thing to check is that multiplication is continuous: $(\colim G_i) \times (\colim G_i) = \colim (G_i \times G_i) \to \colim G_i$ is continuous because finite products commute with colimits in $\U$. It follows immediately that this is the colimit in $\U Ab$.

Clearly in the current situation the colimit topologies on the $A_n$s, $B_n$s and $C_n$s are in $\U$, so the sequence $\colim A_n \to \colim B_n \colim C_n$ is an exact sequence of the underlying groups. We know $\colim$ is right exact because it is a left adjoint, so we only need to check $\colim A_n \to \colim B_n$ is a closed inclusion of spaces. Each $A_n \to B_n$ is a closed inclusion, which is an equaliser in $\U$, so the result follows because finite limits commute with colimits of sequences of closed inclusions in $\U$ by Lemma \ref{finitelim}.
\end{proof}

\begin{thm}
\label{excisionwe}
$LF(\underline{\Sing}(X))$ is weakly equivalent (in the regular structure on $s\U Ab$) to the homotopy colimit of $\{LF(\underline{\Sing}(U))\}_{U \in C'}$ (in the compact Hausdorff structure on $s\U Ab$).
\end{thm}
\begin{proof}
We already know that $LF(\underline{\Sing}(X))$ is weakly equivalent (in the compact Hausdorff structure) to $LF(\underline{\Sing}'(X)) = F(Q(\underline{\Sing}'(X))) = F(Y^\gamma)$. We also have that the homotopy colimit of $\{LF(\underline{\Sing}(U))\}_{U \in C'}$ is $$F(\colim_{U \in C'} Q(\underline{\Sing}(U)) = F(Y^0)$$ by Lemma \ref{hocolim}. We will show inductively that $F(Y^\alpha)$ is weakly equivalent to $F(Y^0)$, in the regular structure, for all $\alpha \leq \gamma$; the result follows.

When $\alpha$ is a limit ordinal this follows immediately from Lemma \ref{colimexact}. When $\alpha$ is a successor and $Y^{\alpha-1} \to Y^\alpha$ is a trivial cofibration, it is trivial. When $\alpha$ is a successor and $Y^{\alpha-1} \to Y^\alpha$ comes from an $n$-dimensional opportunity to glue cells together in the notation used at the beginning of the section, use the Dold-Kan correspondence of \cite[Theorem 1.2.3.7]{LurieHA} and think of the $F(Y^\alpha)$ as chain complexes in $\U Ab$. The mapping cone of $F(Y^{\alpha-1}) \to F(Y^\alpha)$ is homotopic to the complex $M$ with $$M_{n+r} = \bigoplus_{K_{i_1} \cap \cdots \cap K_{i_r} \neq \emptyset} F(K_{i_1} \cap \cdots \cap K_{i_r})$$ with maps induced by the face maps between the $E_{i_1,\ldots,i_r}$, so we just need to show this is exact in the regular structure. It is not hard to check on the elements that the homology groups of the underlying abstract chain complex are trivial; it follows from \cite[Proposition 2.32]{LaMartin} that the induced maps $\coim(M_{k+1} \to M_k) \to \ker(M_k \to M_{k-1})$ are isomorphisms in $\U Ab$, as required.
\end{proof}

\begin{rem}
Pushing out by these collections of cells, which look provocatively like finite length exact sequences of projectives, ought to be a trivial cofibration in some nice homotopical structure on $s\U$, but I do not know what that structure should be.
\end{rem}

From Theorem \ref{excisionwe} we may deduce the usual form of the Excision Theorem as in \cite[p.124]{Hatcher}, using mapping cones (i.e. homotopy cokernels) instead of quotient complexes. 

\begin{thm}[Excision Theorem]
\label{exc}
Given subspaces $A \subseteq B \subseteq X$ in $\U$ with $A$ closed and $B$ open, the inclusion $(X \setminus A, B \setminus A)֓ \to (X,B)$ induces isomorphisms of the homology group objects $H_n(X \setminus A, B \setminus A)֓ \to H_n(X,B)$ for all $n$. Equivalently, for open subspaces $A, B \subseteq X$ covering $X$, the inclusion $(B,A \cap B)֓ \to (X,A)$ induces isomorphisms of homology group objects $H_n(B,A \cap B) \to H_n(X,A)$ for all $n$.
\end{thm}
\begin{proof}
We prove the second formulation. Write $Q(\underline{\Sing}(A+B))$ for the pushout of $Q(\underline{\Sing}(A)) \leftarrow Q(\underline{\Sing}(A \cap B)) \to Q(\underline{\Sing}(B))$. The square
\[
\xymatrix{LF(\underline{\Sing}(A \cap B)) \ar[r] \ar[d] & LF(\underline{\Sing}(B)) \ar[d] \\
LF(\underline{\Sing}(A)) \ar[r] & F \circ Q(\underline{\Sing}(A+B))}
\]
is a pushout, so the cokernels of the two rows are isomorphic, so the maps induced on their homology group objects are isomorphisms. By Theorem \ref{excisionwe}, the canonical map $F \circ Q(\underline{\Sing}(A+B)) \to LF(\underline{\Sing}(X))$ is a weak equivalence in the regular structure, so the maps induced on their homology group objects are isomorphisms too. Now we can use the long exact sequence in homology to see that the maps of homology group objects induced by
\begin{align*}
& \coker(LF(\underline{\Sing}(A)) \to F \circ Q(\underline{\Sing}(A+B))) \\
\to & \coker(LF(\underline{\Sing}(A)) \to LF(\underline{\Sing}(X)))
\end{align*}
are isomorphisms, and the result follows.
\end{proof}

Much of the rest of \cite[Chapter 2]{Hatcher} can be carried across to our topological homology theory fairly painlessly from this point; we give a few highlights.

\begin{thm}[Mayer--Vietoris sequence]
\label{MV}
For open subspaces $A, B \subseteq X$ covering $X$ there is a long exact sequence of homology group objects $$\cdots \to H_{n+1}(X) \to H_n(A \cap B) \to H_n(A) \oplus H_n(B) \to H_n(X) \to \cdots.$$
\end{thm}
\begin{proof}
Think of the rows of the commutative square in the proof of Theorem \ref{exc} as a chain complex of chain complexes; the vertical maps give a map between these chain complexes of chain complexes. This square is a pushout and a pullback; it is a standard result of commutative alebra in quasi-abelian categories that the mapping cone of this square is then a short exact sequence of chain complexes $$0 \to LF(\underline{\Sing}(A \cap B)) \to LF(\underline{\Sing}(A)) \oplus LF(\underline{\Sing}(B)) \to F \circ Q(\underline{\Sing}(A+B)) \to 0$$ in the regular structure. Taking homology and applying Theorem \ref{excisionwe} gives the result.
\end{proof}

In constructing a homology theory on $\U$, we would like certain axioms to be satisfied: those of a generalised homology theory. These axioms are usually listed for homology theories from spaces to abelian groups, but they make sense in our context. A generalised homology theory is here taken to be a functor $E$ from pairs $(X,Y)$ of spaces $Y \subseteq X$ in $\U$ to chain complexes in $R$-$\U Mod$, with the regular structure, satisfying:
\begin{enumerate}[(i)]
\item homotopy invariance: homotopies in $\U$ induce homotopies in $R$-$\U Mod$;
\item exactness: associated naturally to a pair $(X,Y)$ is an exact triangle $E(Y,\emptyset) \to E(X,\emptyset) \to E(X,Y) \to$;
\item additivity: if $(X,A)$ is a disjoint union of pairs $(\bigsqcup X_i, \bigsqcup A_i)$, then the canonical map $\bigoplus E(X_i,A_i) \to E(X,A)$ is a weak equivalence;
\item dimension: $E(\ast,\emptyset)$ is exact in non-zero dimensions;
\item excision: for $A \subseteq B \subseteq X$ with $A$ closed and $B$ open, the canonical map $E(X \setminus A, B \setminus A) \to E(X,A)$ is a weak equivalence.
\end{enumerate}

Given our definitions, the excision axiom was the only non-trivial thing to check. So we have:

\begin{thm}
\label{generalised}
Singular homology $LF \circ \underline{\Sing}$ is a generalised homology theory.
\end{thm}

\section{Eilenberg--Mac Lane Spaces}
\label{EM}

Now we will use Theorem \ref{singhoco} to construct certain Eilenberg--Mac Lane spaces for topological groups.

\begin{lem}
For all $n$, $\underline{\Sing} \circ \lvert \Delta^n \rvert$ is weakly equivalent to $\Delta^n$ in the compact Hausdorff structure.
\end{lem}
\begin{proof}
Any map $\Delta^0 \to \Delta^n$ is a trivial cofibration; $\lvert \Delta^n \rvert$ is homotopy equivalent to a point, so $\underline{\Sing} \circ \lvert \Delta^n \rvert$ is homotopy equivalent to $\Delta^0$.
\end{proof}

Suppose $X \in s\U$. As usual, each $Q(X)_n$ comes with a canonical decomposition as a disjoint union of compact fat cells: write $S$ for the set of fat cells. For each fat $n$-cell $K \in S$, consider the diagram $D_K$ in $s\U$ given by objects $K \times \Delta^m$ indexed by the $m$-cells of $\Delta^n$ and maps indexed by the face and degeneracy maps of $\Delta^n$. Write $Q(X)$ as the colimit of a sequence $(Y^\alpha)$ of pushouts $Y^\alpha \leftarrow K \times \partial\Delta^n \to K \times \Delta^n$ attaching fat cells in order of dimension: these attaching maps induce maps from the copies of $K \times \Delta^m$ in $D_K$ to copies of $K' \times \Delta^m$ in $D_{K'}$ wherever an $m$-face of $K$ is attached to an $m$-face of $K'$. The set of diagrams $D_K$, $K \in S$, together with these maps between objects of the diagrams, give a bigger diagram $D$.

\begin{prop}
$Q(X)$ is the homotopy colimit of $D$.
\end{prop}
\begin{proof}
The colimit of $D$ is the colimit of the $Y^\alpha$, since we can get each $K \times \partial\Delta^n$ as the colimit (and in fact the homotopy colimit) of all the objects of $D_K$ except $K \times \Delta^n$ itself, and then the maps between the $D_{K'}$s induce the map $Y^\alpha \leftarrow K \times \partial\Delta^n$ from the pushout. And we know the colimit of the $Y^\alpha$ is $Q(X)$.

Since we are adding fat cells in order of dimension, for each $n$ there is some $\alpha_n$ such that $Y^{\alpha_n}$ is the $n$-skeleton of $Y$. Then $Y^{\alpha_n}$ is the homotopy colimit of the pushout $$Y^{\alpha_{n-1}} \leftarrow \bigsqcup_{\alpha_{n-1} \leq \alpha < \alpha_n} K_\alpha \times \partial\Delta^n \to \bigsqcup_{\alpha_{n-1} \leq \alpha < \alpha_n} K_\alpha \times \Delta^n$$ because all three objects are cofibrant and the second map is a cofibration, by \cite[Proposition A.2.4.4]{LurieHTT}. Then $Y$ is the colimit of a sequence $(Y^{\alpha_n})$ of cofibrant objects with cofibrations between them, so it is the homotopy colimit by \cite[Example 11.5.11]{Riehl}. It follows that colimit $Q(X)$ is the homotopy colimit of $D$.
\end{proof}

Since we can change a diagram by a weak equivalence without changing the homotopy colimit, we can change $D$ to a diagram $D'$ of the same shape where we replace each $K \times \Delta^m$ with $\disc(K) = K \times \Delta^0$; the maps in $D'$ are the obvious ones induced by those of $D$. We conclude that $Q(X)$ is the homotopy colimit of $D'$.

We denote the points of $D^n \times X_n$ by the coordinates $(r,\theta,x)$, where $(r,\theta)$ parametrises the closed unit ball in spherical coordinates and $x \in X_n$. For each $n$, we have a map $\Phi^n$ from $D^n \times X_n$ to the $n$-skeleton $\lvert X \rvert^n$ of $\lvert X \rvert$ which restricts to a homeomorphism from $(D^n \setminus S^n) \times X_n$ to $\lvert X \rvert^n \setminus \lvert X \rvert^{n-1}$, in the same way as for CW-complexes.

\begin{thm}
\label{tpd}
Suppose $X \in s\U$ with $X_n$ totally path-disconnected for all $n$. Then $\underline{\Sing} \circ L\lvert X \rvert$ is weakly equivalent to $X$ in the regular structure.
\end{thm}
\begin{proof}
By replacing $X$ with $\Ex(X)$ (which gives a weak equivalence in the compact Hausdorff structure), if necessary, we may assume for simplicity that for any fat cell in $Q(X)$ the images of its face maps are pairwise disjoint.

$Q(X)$ is weakly equivalent to $X$ in the compact Hausdorff structure, and if the $X_n$ are totally path-disconnected, so are the $Q(X)_n$.

$\lvert Q(X) \rvert$ is then a KW-complex, as defined in \cite[Section 2]{Myself}, and we construct an open cover inductively on the skeleta: on the $0$-skeleton we take an open cover whose open sets are the fat cells of $Q(X)_0$. Suppose we have an open cover $C_{n-1}$ of the $(n-1)$-skeleton. For each fat cell $K$ of $Q(X)_n$, we have the open space $U_K$ of open $n$-cells $K \times (D^n \setminus S^n)$. Also, for every open $U \in C_{n-1}$, think of $(\Phi^n)^{-1}(U)$ as an open subset of $S^n \times X_n$; in spherical coordinates as before, define $$U' = (1 - \delta,1] \times (\Phi^n)^{-1}(U) \cup U$$ for some small $\delta > 0$. Take $C_n$ to be the the open cover of the $n$-skeleton given by the $U_K$s and the $U'$s. Taking limits over $n$, we get a cover $C$ of $\lvert Q(X) \rvert$.

Let $C'$ be the set of finite intersections of open sets in $C$. We can see directly that any non-empty element of $C'$ is homotopic to some fat cell $K$ of $Q(X)$. Since $K$ is totally path-disconnected, all singular maps into $K$ are constant and $\underline{\Sing}(K)$ is just the constant simplicial space on $K$. By Theorem \ref{singhoco}, $\underline{\Sing} \circ L\lvert X \rvert$ is weakly equivalent (in the regular structure on $s\U$) to the homotopy colimit of $\{\underline{\Sing}(U)\}_{U \in C'}$, which is the homotopy colimit of $\{\disc(K)\}_{U \in C'}$. But this diagram is easily seen to be isomorphic to $D'$, so its homotopy colimit is $X$.
\end{proof}

Given a group object $G$ in $s\U$, a construction is given in \cite[Section 6]{Myself} for $\bar{W}G \in s\U$ such that $\pi_n(\bar{W}G) = \pi_{n-1}(G)$ for $n > 0$, and $\pi_0(\bar{W}G) = \{\ast\}$. In this construction $\bar{W}G_n = G_n \times \cdots \times G_0$.

Now if $G$ is a topological group in $\U$, $\disc G$ is a group object in $s\U$, and $\pi_n(\bar{W}G)$ is $G$ for $n = 1$ and trivial otherwise. Note too that if $G$ is totally path-disconnected, every $\bar{W}G_n$ is too. So we can apply Theorem \ref{tpd}.

\begin{thm}
\label{EMspace}
If $G$ is totally path-disconnected, $L\lvert \bar{W}G \rvert$ is an Eilenberg--Mac Lane space $K(G,1)$ for $G$.
\end{thm}

Similarly, when $G$ is totally path-disconnected and abelian, $L\lvert \bar{W}^nG\rvert$ is a $K(G,n)$.

Note that totally disconnected spaces are totally path-disconnected, so totally disconnected groups are included as a special case. Here is a totally path-disconnected topological group in $\U$ which is not totally disconnected.

\begin{example}
\label{freepseudo}
Let $X$ be the pseudo-arc, as in Example \ref{pseudo}. Let $F(X)$ be the free group object (in $\U$) on $X$: see \cite[Notation 2.10]{LaMartin} for the construction. $F(X)$ is not totally disconnected because $X$ is connected and the canonical map $X \to F(X)$ is a closed embedding by \cite[Theorem 2.12]{LaMartin}.

Suppose there is a non-constant map $p: [0,1] \to F(X)$. In the notation of \cite{LaMartin}, topologically $F(X)$ is the colimit of the closed subspaces $\pi(\bigsqcup_{i \leq n} M^i)$ by \cite[Proposition 2.16]{LaMartin}. So by compactness, there is some $n$ such that the image of $p$ is contained in $\pi(\bigsqcup_{i \leq n} M^i)$ but not $\pi(\bigsqcup_{i \leq n-1} M^i)$; in particular there is some open subspace $$U = p^{-1}(\pi(\bigsqcup_{i \leq n} M^i) \setminus \pi(\bigsqcup_{i \leq n-1} M^i))$$ of $[0,1]$ on which $p$ is not constant, or else for $x \in \bar{U}$ we would have $p(x) \in \pi(\bigsqcup_{i \leq n} M^i) \setminus \pi(\bigsqcup_{i \leq n-1} M^i)$, implying $U = \bar{U} = [0,1]$ and $p$ is not constant on $[0,1]$ by hypothesis. Then pick a closed interval $[a,b] \subseteq U$ on which $p$ is not constant.

From the construction, $M^i$ is the disjoint union of spaces of the form $$M_S = \{(x_1, \dots, x_n) \in (X \sqcup X^\ast)^n: x_i \in X \text{ for } i \in S, x_i \in X^\ast \text{ for } i \notin S\}$$ where $S$ ranges over subsets of $\{1, \ldots, n\}$. Observe that if we restrict to the subspaces $$M'_S = M_S \cap \pi^{-1}(\pi(\bigsqcup_{i \leq n} M^i) \setminus \pi(\bigsqcup_{i \leq n-1} M^i)),$$ we have $$\pi(\bigsqcup_{i \leq n} M^i) \setminus \pi(\bigsqcup_{i \leq n-1} M^i) = \bigsqcup_{S \subseteq \{1, \ldots, n\}} \pi(M'_S).$$ So $p([a,b])$ must be contained in some $\pi(M'_S)$.

It follows from \cite[Proposition 2.16]{LaMartin}, by the same argument as \cite[Corollary 2.18]{LaMartin}, that $\pi: M'_S \to F(X)$ is an embedding. So $p|_{[a,b}$ lifts to a non-constant map $p: [a,b] \to M'_S$, but $M'_S$ is totally path-disconnected as a subspace of a product of copies of $X$, giving a contradiction, and implying that $F(X)$ is totally path-disconnected.
\end{example}

As a final application, we return to singular homology. Generalise the definition of $\Delta$-complexes given in \cite[p.103]{Hatcher} to allow totally path-disconnected spaces of cells: as in the classical case, the definition ensures that such a $\Delta$-complex $X' \in \U$ is the geometric realisation of some $X \in s\U$ such that every $X_n$ is totally path-disconnected. Call such spaces generalised simplicial complexes. As in the classical case, we define the simplicial homology $H^\text{Simp}(X')$ of $X'$ to be the singular homology of $X$.

\begin{prop}
The singular and simplicial homology theories for generalised $\Delta$-complexes are naturally isomorphic.
\end{prop}
\begin{proof}
We showed in the proof of Theorem \ref{tpd} that $X$ is weakly equivalent, in the compact Hausdorff structure, to $\{\underline{\Sing}(U)\}_{U \in C'}$, where $C'$ is the open cover of $\lvert X \rvert$ described there. So $LF(X)$ is weakly equivalent in the compact Hausdorff structure to $LF(\{\underline{\Sing}(U)\}_{U \in C'})$, which is weakly equivalent to $LF(\underline{\Sing}(\lvert X \rvert))$ in the regular structure by Theorem \ref{excisionwe}. So all three have the same homology.
\end{proof}


\begin{thebibliography}{99}

\bibitem{Bing}
R.H. Bing, Concerning hereditarily indecomposable continua, Pacific J. Math. Volume 1, Number 1 (1951), 43-51.

\bibitem{Myself}
Corob Cook, G: \emph{Homotopical Algebra in Categories with Enough Projectives}, available at \url{https://arxiv.org/abs/17003.00569}.

\bibitem{Fabel}
Fabel, P: \emph{Multiplication is discontinuous in the Hawaiian earring
group (with the quotient topology)}. Bull. Polish Acad. Sci. Math. 59 (2011), 77-83

\bibitem{Hatcher}
Hatcher, A: \emph{Algebraic Topology}. Cambridge University Press, Cambridge (2002).

\bibitem{Hovey}
Hovey, M: \emph{Model Categories} (Mathematical Surveys and Monographs, 63). Amer. Math. Soc., Providence (1999).

\bibitem{LaMartin}
LaMartin, W: \emph{On the Foundations of $k$-Group Theory} (Diss. Math., CXLVI). Pa\'{n}stwowe Wydawn. Naukowe, Warsaw (1977).

\bibitem{Lewis}
Lewis, L: \emph{The stable category and generalized Thom spectra: Appendix A}, available at \url{http://www.math.uchicago.edu/~may/MISC/GaunceApp.pdf}. Ph.D. Thesis, University of Chicago (1978).

\bibitem{Low}
Low, Z: \emph{Internal and Local Homotopy Theory} (2014), available at \url{https://arxiv.org/pdf/1404.7788}.

\bibitem{LurieHA}
Lurie, J: \emph{Higher Algebra} (2017), available at \url{http://www.math.harvard.edu/~lurie}.

\bibitem{LurieHTT}
Lurie, J: \emph{Higher Topos Theory} (Annals of Mathematics Studies, 170). Princeton University Press, Princeton (2009).

\bibitem{Moerdijk}
Moerdijk, I: \emph{Classifying Spaces and Classifying Topoi} (Lecture Notes in Mathematics, 1616). Springer, Berlin (1995).

\bibitem{Munkres}
Munkres, J: \emph{Topology}, 2nd edition. Prentice Hall (2000).

\bibitem{Riehl}
Riehl, E: \emph{Categorical Homotopy Theory} (New Mathematical Monographs, 24). Cambridge University Press, Cambridge (2014).

\bibitem{Schneiders}
Schneiders, J-P: \emph{Quasi-Abelian Categories and Sheaves}. M\'{e}m. Soc. Math. Fr. (2), 76 (1999).

\bibitem{Segal}
Segal, G: \emph{Classifying Spaces and Spectral Sequences}. Publications Mathématiques de l'IH\'{E}S 34 (1968), 105-112.

\bibitem{Strickland}
Strickland, N: \emph{The Category of CGWH Spaces}, available at \url{https://neil-strickland.staff.shef.ac.uk/courses/homotopy/cgwh.pdf}.

\end{thebibliography}
\end{document}